\documentclass[12pt, reqno]{amsart}

\usepackage{amsmath}
\usepackage{amsthm}
\usepackage{amsfonts}
\usepackage{amssymb}
\usepackage{verbatim}
\usepackage{graphicx}
\usepackage[margin=1.0in]{geometry}
\usepackage{url}
\usepackage{enumerate}
\usepackage[pdftex,bookmarks=true]{hyperref}
\usepackage[normalem]{ulem}
\usepackage[usenames, dvipsnames]{color}
 
\theoremstyle{plain}
  \newtheorem{theorem}{Theorem}[section]

  \newtheorem{lemma}[theorem]{Lemma}

  \newtheorem{claim}[theorem]{Claim}

\theoremstyle{definition}

  \newcommand\C{{\mathbb{C}}}
  \newcommand\N{{\mathbb{N}}}
  
  \newcommand\R{{\mathbb{R}}}
\renewcommand\P{{\mathbf{P}}}
\newcommand\E{{\mathbf{E}}}  
\newcommand\Var{\mathbf{Var}}

\renewcommand\Im{{\operatorname{Im}}}
\renewcommand\Re{{\operatorname{Re}}}
\newcommand{\ep}{\varepsilon}

\newcommand{\del}{\delta}
\newcommand{\per}{\text{per}}

  \newcommand\CH{{\mathcal{H}}}

\begin{document}

 \title{Hole radii for the Kac polynomials and derivatives}
 \author{Hoi H. Nguyen and Oanh Nguyen}
 \address{Department of Mathematics\\ The Ohio State University \\ 231 W 18th Ave \\ Columbus, OH 43210 USA}
 \email{nguyen.1261@osu.edu}
 \address{Division of Applied Mathematics\\ Brown University\\  Providence, RI 02906, USA}
 \email{oanh\_nguyen1@brown.edu}
 \thanks{H.N. is supported by NSF CAREER grant DMS 1752345, O.N. is supported by NSF grants DMS–1954174 and DMS–2246575. This work was initiated under the SQuaREs 2021 program of AIM, we thank the Institute for the generous support.}

\begin{abstract}
The Kac polynomial
$$f_n(x) = \sum_{i=0}^{n} \xi_i x^i$$ with independent coefficients of variance 1 is one of the most studied models of random polynomials.

It is well-known that the empirical measure of the roots converges to the uniform measure on the unit disk. On the other hand, at any point on the unit disk, there is a hole in which there are no roots, with high probability. In a beautiful work \cite{michelen2020real}, Michelen showed that the holes at $\pm 1$ are of order $1/n$. We show that in fact, all the hole radii are of the same order. The same phenomenon is established for the derivatives of the Kac polynomial as well.
\end{abstract}
 
  \maketitle
  \section{Introduction}

Approximation by roots of polynomials of coefficients $\{-1,0,1\}$ is a classical and interesting topic in analysis, with fascinating pictures and conjectures. For instance, it follows from a result of Borwein and Pinner \cite[Theorem 1]{BoPi} that for any given $\zeta$ of $d$-th root of unity, the distance from it to any root $z$ of $\{-1,0,1\}$ polynomials of degree $n$ which vanish at $\zeta$ of order at most $k$ (i.e. $f^{(k+1)}(\zeta)\neq 0$) can be bounded by 
$$|z-\zeta| \ge e^{-1} \frac{(k!)^{\lceil \phi(d)/2\rceil }}{(n+1)^{ (k+1)\lceil \phi(d)/2\rceil+1 }},$$ 
where $\phi(d)$ is the usual Euler phi-function. This result is asymptotically optimal. On the other hand, the smallest distance can be sub-exponentially small if $\zeta$ is on the unit circle and not a root of unity (such as when $\zeta$ is an algebraic number of small Mahler measure), see for instance \cite[Corollary 1, Theorem 3]{BoPi}. The situation at 1 is also interesting, it was shown from the same paper \cite[Corollary 4, Theorem 6]{BoPi} that $|1-z| \ge  \frac{1}{n^{k+2}}$ (which is again near optimal) for any real roots $z$ of $\{-1,0,1\}$ polynomials of degree $n$ which vanish at $1$ of order exactly $k$. Note that the distance is significantly larger if $z$ is purely complex. See Figure \ref{fig:BP}. We also refer the reader to \cite{BoPi,CKW, OP} and the references therein for further interesting discussions and problems. 
\begin{figure}[h!] \label{fig:BP}
	\includegraphics[width=.5\textwidth]{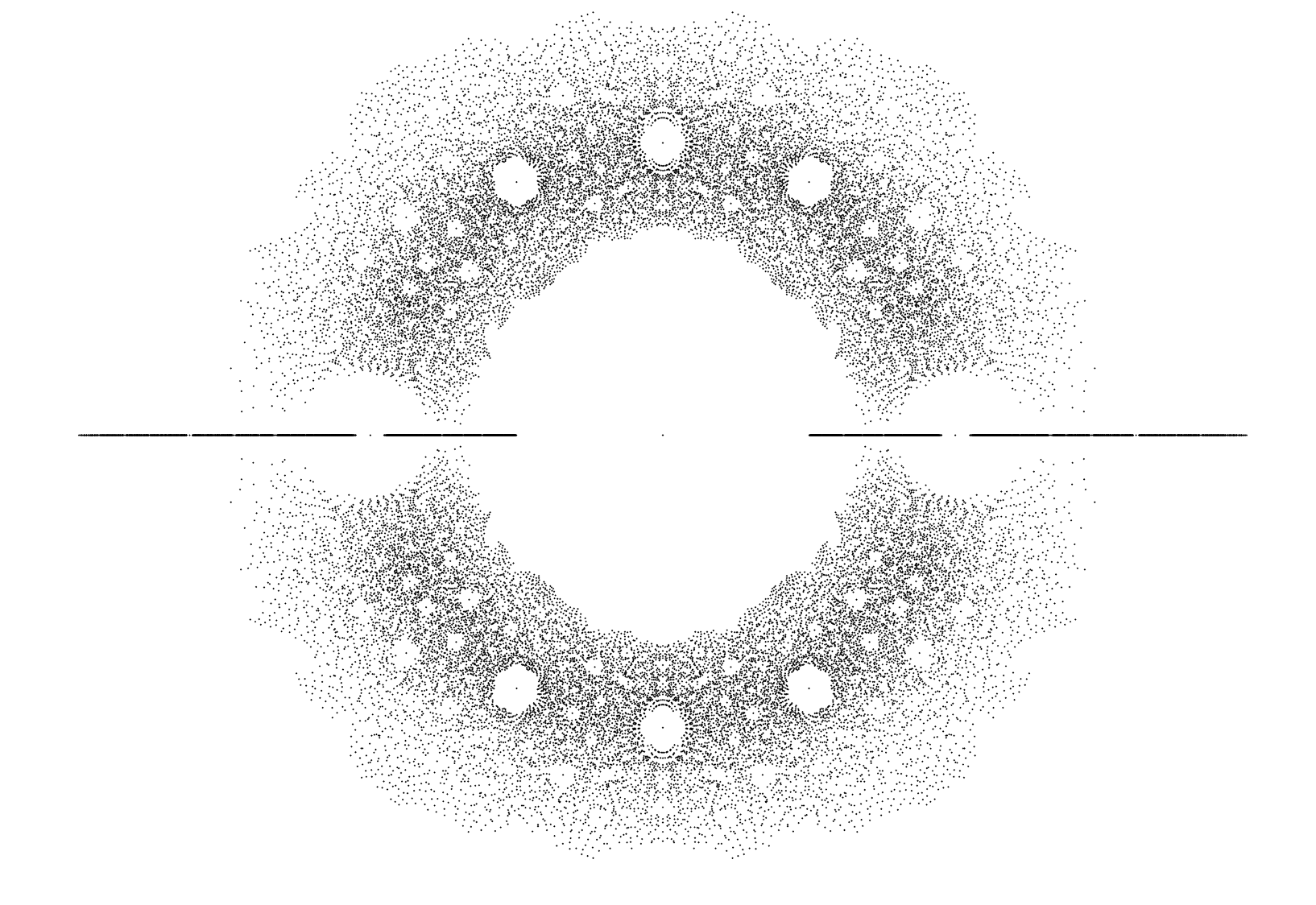} 
	\caption{Zeros of all polynomials with $\pm 1$ coefficients and degree at most eight (\cite{BoPi})}
\end{figure}
\vspace{5mm}

Our goal in this note is to study the distances from some probabilistic viewpoint. More generally, consider the Kac polynomial
$$f_{0, n}(x) = \sum _{i=0}^{n} \xi_i x^{i}$$
where $\xi_i$ are independent (not necessarily identically distributed), real-valued random variables with mean 0 and variance 1. 

\vspace{5mm}

For this random polynomial, it is well-known that the empirical distribution of the roots converges to the uniform distribution on the unit circle (\cite{kabluchko2014asymptotic}). So, the roots concentrate near the unit circle. And in particular, the real roots concentrate near $\pm 1$. However, precisely at $\pm 1$, there are holes that do not contain any roots. 
It was conjectured by Shepp and Vabderbei \cite{shepp1995complex}, and confirmed recently by Michelen \cite[Theorem 1.2]{michelen2020real} that the typical distance of {\it real roots} to 1 for random Kac polynomial is of order $O(1/n)$.

\begin{theorem}\label{thm:hole1} Let $\xi_i$ be iid with mean zero and variance one. For any constant $\delta>0$, there exists a constant $C>0$ so that 
$$\P(\mbox{there exists a real root in $[1-C/n, 1+C/n]$}) \ge 1-\delta$$ 
for all $n$ sufficiently large.
\end{theorem} 
\vspace{5mm}
It is not hard to establish the lower bound and conclude that the hole at $1$ (and $-1$) has radius of order $\Theta(1/n)$.

How about the hole radius at other points on the unit circle? A recent result by Cook, Yakir, Zeitouni and the first author \cite{cook2023universality} (see also Michelen and Sahasrabudhe \cite{michelen2020random} for the Gaussian case) shows that the distance between the zero set of $f_{n, 0}$ and the unit circle is of order $\frac{1}{n^2}$. So, this is a lower bound for all hole radii. 

From the result in \cite{BoPi} and Figure 1, it is natural to predict that the hole radii exhibit different orders at different points. For instance, in Figure 2 (source \cite{Baez}) where all roots of polynomials with coefficients $\pm 1$ and degree at most 24, one can observe that there are largest holes at $\pm 1$, smaller holes possibly at the roots of unity, and barely visible holes at other points. In Figure 3, we draw sampled roots of random Kac polynomials with $\pm 1$ coefficients and degree $n=1000$. Note that as $n=1000$ is large compared to $24$,  the holes are no longer visible in the figure unless being zoomed in properly. However, the striking similarities between the two figures would suggest that the same observation would remain true for large $n$. In Figure 4, we display sampled roots of the first derivative polynomial with $\pm 1$ coefficients and degree $n=1000$ which has the same pattern as in Figure 3.

\begin{figure}[h!] \label{fig:baez}
	\includegraphics[width=.5\textwidth]{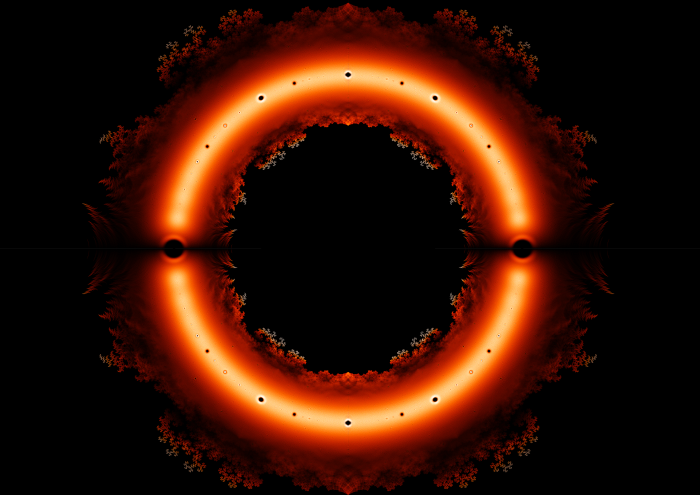} 
	\caption{All roots of polynomials with coefficients $\pm 1$ and degree at most 24 (source \cite{Baez})}
\end{figure}

\begin{figure}[h!]\label{fig:Kac}
	\includegraphics[width=.55\textwidth]{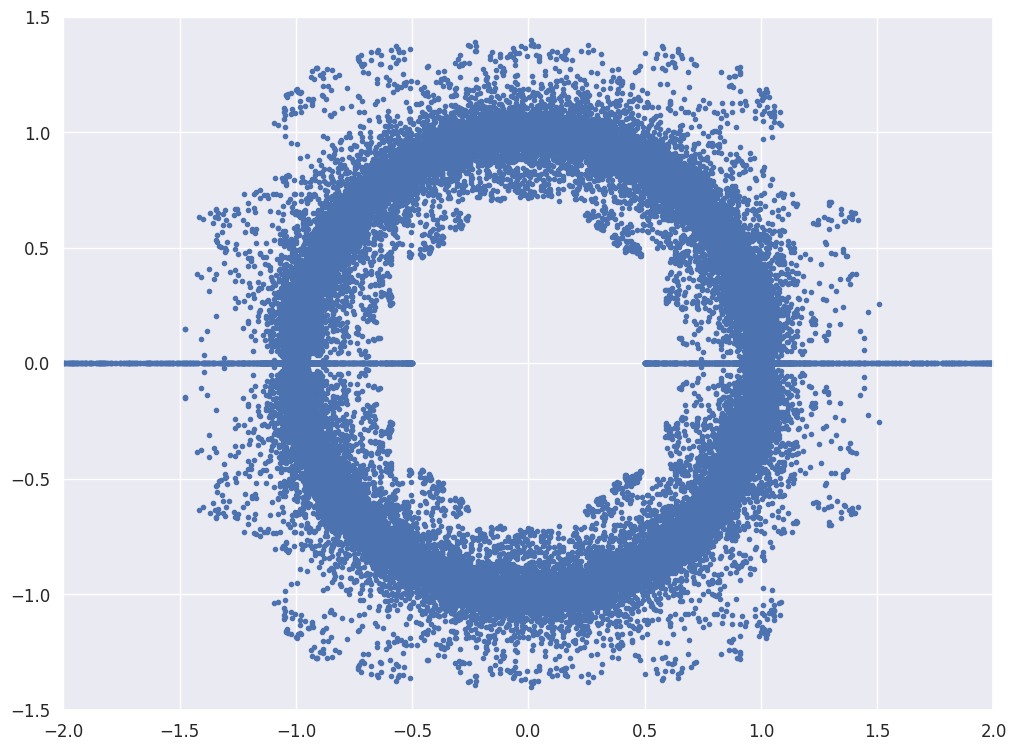}
	\caption{Sampled roots of the Kac polynomials with Rademacher coefficients}
\end{figure}

\begin{figure}[h!]\label{fig:derivative}
	\includegraphics[width=.55\textwidth]{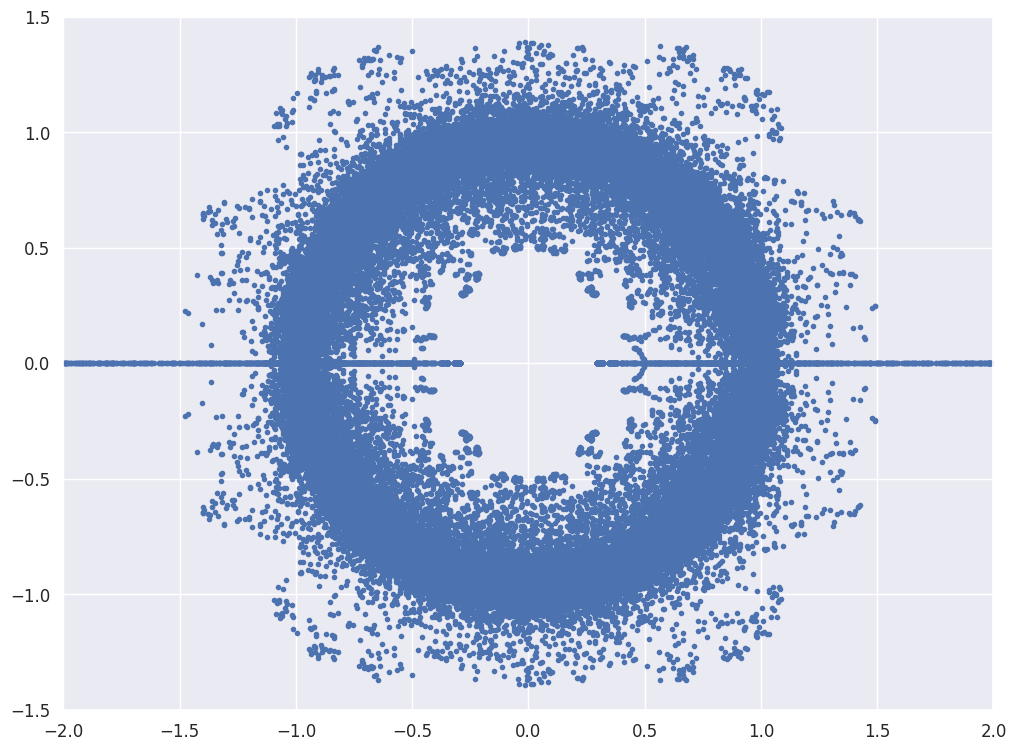}
	\caption{Sampled roots of the first derivative of Kac polynomials with Rademacher coefficients}
\end{figure}


Disproving this prediction, in this paper, we show that for every point $\zeta$ on the unit circle, the precise order of the hole at $\zeta$ is $\Theta(1/n)$. Moreover, we show that this holds also for the derivatives of the Kac polynomial. For a positive integer constant $\rho$, let us define
$$f_{\rho, n} = \sum_{i=0}^{n} a_{i, \rho, n} \xi_i z^i$$
where $a_{i, \rho, n} = (1+o_{n}(1)) i(i-1)\dots(i-\rho+1)\in \R$ with the $o_{n}(1)$ converges to $0$ uniformly in $i$, as $n\to\infty$ and $\rho$ fixed.
For any set $S\subset \C$, let $N_{f_{\rho, n} }(S)$ be the number of roots of $f_{\rho, n}$ in $S$. If $\rho=0$, we get the Kac polynomial. The $\rho$-th derivative of the Kac polynomial corresponds to $x^{-\rho}f_{\rho, n}(x)$. Here is our main result.
 \begin{theorem}  \label{thm:main}
Let $\rho$ be any positive integer constant. Assume that the random variables $\xi_i$ are independent with mean 0, variance 1 and bounded $(2+\ep_0)$-moment for some $\ep_0>0$. For every $\zeta\in S^{1}$, the radius of the hole at $\zeta$ is $\Omega(1/n)$. In particular, for every $\ep>0$, there exist positive constants $c_\rho$ and $C_\rho$ such that for all $\zeta\in S^{1}$		
\begin{equation}\label{eq:hole:upperbound}
\text{[Upper bound]}\hspace{20mm}	\P(N_{f_{\rho, n}}(B(\zeta, C_{\rho}/n))>0)\ge 1-\ep
	\end{equation}
	and 
	\begin{equation}\label{eq:hole:lowerbound}
\text{[Lower bound]} \hspace{20mm}	\P(N_{f_{\rho, n}}(B(\zeta, c_{\rho}/n))=0)\ge 1-\ep
	\end{equation}
	for sufficiently large $n$.  
\end{theorem}

As far as we understand, the proof in \cite{michelen2020real} is restricted to real roots and cannot be applied to complex roots.

For a related discussion concerning the hole probability, we refer to \cite{buckley2018hole}, \cite[Section 5.1.3]{HKPV} and the references therein.

\section{Proof sketch and ingredients}
For the upper bound, we need to split into two cases: $\zeta = \pm 1$ and $\zeta\neq \pm 1$. For the former, we show that the proof in \cite{michelen2020real} can be adapted to  cover $f_{\rho, n}$ for general $\rho\in \N$. This method relies on the simple observation that if a polynomial $f$ changes sign in an interval on the real line then it has at least one root there. Since this observation only holds for real roots, it merely works for $\zeta = \pm 1$ where one can reduce the upper-bound problem to showing that there exists a real root in the interval centered at $\zeta$ and radius $C/n$. For $\zeta\neq \pm 1$, we need a different approach.

To this end, we note that the expected value of $N_{f_{\rho, n}}(B(\zeta, C/n))$ being large does not imply that the number of real roots is non-zero with high probability. However, it can be achieved via Chebyshev inequality if we can show that its variance is of smaller order than its mean squared. To do so, the high-level idea is to show that $f_{\rho, n}$, after rescaled properly, converges to a Gaussian process, say $f_{\infty}$ which we can estimate the growth of $\Var \left (N_{f_{\infty}}(B(\zeta, C/n))\right )$ in terms of $C$ and then pass the result back to $f_{\rho, n}$. So, we consider a rescaled version of $f_{\rho, n}$ by zooming in at the local neighborhood of $\zeta$ as follows
$$g_{n}(z)=\frac{1}{n^{\rho+1/2}}f_{\rho, n}\left (\zeta+\frac1n \zeta z\right ).$$
When the random variables are Gaussian, we know that $g_n$ is a Gaussian process with covariance 
\begin{eqnarray}
\E g_n(z)\overline{g_n(w)} &=& \frac{1}{n^{2\rho+1}} \sum_{k=0}^{n} |\zeta|^{2k} a_{k, \rho, n}^{2}\left (1+\frac1n z\right )^{k}\left (1+\frac1n \bar w\right )^{k}\nonumber\\
 &=&  \frac{1}{n^{2\rho+1}}  \sum_{k=0}^{n} a_{k, \rho, n}^{2}\left (1+\frac1n z\right )^{k}\left (1+\frac1n \bar w\right )^{k}.\nonumber
\end{eqnarray}
We note that $\zeta$ disappears on the right-most side and could potentially account for why the hole radii are of the same order.

\vspace{5mm}
To show the cancellation $\Var \left (N_{f_{\infty}}(B(\zeta, C/n))\right )=o\left (\E \left (N_{f_{\infty}}(B(\zeta, C/n))\right )\right )^{2}$, we note that the variance is an integral of $\rho_2(z_1, z_2) - \rho_1(z_1) \rho_1(z_2)$ which is in turn of order $|z_1 - z_2|^{-1}$. So, when $z_1$ and $z_2$ are far away, the integrand is small, accounting for the cancellation. This suggests that the numbers of roots in far-away regions are weakly correlated which is consistent with previously established results for real roots (\cite{nguyenvuCLT}) and radius of complex roots (\cite{cook2023universality}). To handle the diagonal region when $z_1$ is near $z_2$, we come up with a simple argument, though via rather long and tedious algebraic manipulations, showing that the function $\rho_2(z_1, z_2) - \rho_1(z_1) \rho_1(z_2)$ is indeed continuous everywhere and hence the diagonal has negligible contribution, see Lemma \ref{lm:discont}. To carry out this strategy, we actually replace the whole ball $B(\zeta, C/n)$ by a subset, denoted by $U_C$, which is a thin strip along the unit circle. This is a major device that allows us to reduce from $|z_1-z_2|$ to $|\Im(z_1) - \Im(z_2)|$ which reduces the dimension and facilitates the rather elegant proof that follows. 

\vspace{5mm}
To pass from $f_{\infty}$ to $f_{\rho, n}$, we need to show some sort of uniform integrability of $N_{f_{\infty}}(B(\zeta, C/n)$. For that, we adapt a double Taylor expansion argument used in \cite{krishnapur2021number} (see Lemma \ref{lemma:moments:infty}).

\vspace{5mm}
Finally, to establish the lower bound \eqref{eq:hole:lowerbound}, we will show that the expected number of roots in the ball $B(\zeta, c/n)$ is small for sufficiently small $c$ and then apply Markov's inequality. The derivation of the expected number of roots is first reduced to the Gaussian setting when all coefficients $\xi_i$ are iid standard Gaussian, via the {\it universality} properties of the random polynomials. To do the calculation for the Gaussian case, there are two possible ways. The first way is to directly apply the classical Kac-Rice formula to $f_{\rho, n}$. The second way, which is what we perform here, is to derive it through $f_\infty$ using the limits that we already establish for the upper bound.

\vspace{5mm}
{\bf Notations.} For the rest of the paper, to simplify the notation, we will often drop the subscript $\rho$. For instance, we write $f_n$ in place of $f_{\rho, n}$. For a function $f$, let $\mathcal Z(f)=\{z\in \C: f(z) = 0\}$ be the zero set of $f$. 

We use standard asymptotic notations under the assumption that $n$ tends to infinity.  For two positive  sequences $(a_n)$ and $(b_n)$, we say that $a_n \gg b_n$ or $b_n \ll a_n$ if there exists a constant $C$ such that $b_n\le C a_n$. If $|c_n|\ll a_n$ for some sequence $(c_n)$, we also write $c_n\ll a_n$.

If $a_n\ll b_n\ll a_n$, we say that $b_n=\Theta(a_n)$.  If $\lim_{n\to \infty} \frac{a_n}{b_n} = 0$, we say that $a_n = o(b_n)$.  We also write that $a_n=O_C(b_n)$ if the implied constant depends on a given parameter $C$.


\section{Proof of Theorem \ref{thm:main}: upper bound for $\zeta = \pm 1$}\label{sec:upperbound0}
When $\zeta = \pm 1$, Michelen \cite{michelen2020real} already showed the stated upper bound for the Kac polynomial $f_{0, n}$. We will show that this proof can be easily adapted to cover the general case $f_{\rho, n}$. We assume that $\zeta=1$ as the case $\zeta=-1$ is completely similar. It suffices to show that with probability $\ge 1-\ep$, there is at least one root of $f_{\rho, n}$ in the interval $J := [1-C/n, 1+C/n]$ for some large constant $C$. 
Let $f_{0, n}$ be the $\rho$-th anti-derivative of $z^{-\rho}f_{\rho, n}$, then 
$$f_{0, n} = \sum_{i=0}^{n} (1+o(1)) \xi_i z^i$$
is basically the Kac polynomial (if disregarding the $1+o(1)$ terms). By interlacing, this can be deduced from showing that there are at least $\rho+1$ roots of $f_{0, n}$ in the same interval.  To this end, we show that we can find $(\rho+1)$ sub-intervals of $J$ each of which observes a sign change of $f_{0, n}$ and hence contains at least one root. 

Consider the rescaled polynomial
$$h_n(x) = \frac{1}{\sqrt n} f_{0, n}(1+x/n), \quad x\in \R.$$
Let $M$ be a large constant and $x_1, \dots, x_M$ be deterministic points in $J$. By \cite[Lemma 5]{michelen2020real} (which is a rather direct application of the Lindeberg-Feller Central Limit Theorem), the random vector $(h_n(x_1), \dots, h_n(x_M))$ converges to the Gaussian vector $(h(x_1), \dots, h(x_M))$ where $h$ is a centered, real Gaussian process with covariance
$$\E h(x) h(y) = \int_{0}^{1} e^{(x+y)t} dt.$$
(In fact, in \cite{michelen2020real}, this result is established for the Kac polynomial without the $1+o(1)$ terms as above but the proof can easily go through without changes when these terms are present.)

Let $M = (\rho+1)K$ where $K$ is a large constant to be chosen. By \cite[Lemma 6]{michelen2020real}, there exists a constant $\gamma>0$ such that any centered Gaussian vector $(Z_1, \dots, Z_K)$ with variances $\E Z_i^{2}=1$ for all $i$ and covariances $|\E Z_i Z_j|\le \gamma$ for all $i\neq j$ satisfies 
$$\P(\text{all $Z_1, \dots, Z_K$ have the same sign})\le 2^{-K+2}.$$

Direct calculation shows that if $y = \alpha x$ with $x>1$ and $\alpha>1$ then 
$$\left |\E \frac{h(x) h(y)}{\sqrt{\Var h(x)\Var h(y)}}\right | = o_{\alpha\to \infty} (1).$$
So, for a given $K$, by taking $\alpha$ sufficiently large, we can make this number smaller than $\gamma$. We then take $x_i = \alpha^{i-1}$, $i=1, \dots, M$. So, for $j=0, \dots, \rho$,
 $$\P(\text{all $h(x_{jK+1}), \dots, h(x_{jK+K})$ have the same sign})\le 2^{-K+2}.$$
 And so, for sufficiently large $n$,
  $$\P(\text{$h_n$ does not have any real roots in $[1+x_{jK+1}/n, 1+x_{jK+K}/n]$})\le 2^{-K+3}.$$
  By the union bound, the probability that $h_n$ has less than $\rho+1$ real roots in $[1+x_{1}/n, 1+x_{M}/n]$ (which is a union of $\rho+1$ such intervals above) is at most $(\rho+1) 2^{-K+3}$. By choosing $K$ sufficiently large so that this number is smaller than $\ep$, we obtain the desired tail probability.

\section{Proof of Theorem \ref{thm:main}: upper bound for $\zeta\neq \pm 1$}\label{sec:upperbound}
We want to show that there exists a constant $C$ such that with probability at least $1-\ep$, there is at least one root of $f_{\rho, n}$ in the ball $B(\zeta, 2C/n)$. For a sufficiently small constant $\delta>0$ depending only on $\rho$, we consider the strip $\zeta+\zeta U_C/n$ that goes along the unit circle where $U_C=(-\delta, \delta)\times (-C, C)\subset B(0, 2C)$. Since this strip is a subset of $B(\zeta, 2C/n)$, it suffices to show that with probability at least $1-\ep$, there is at least one root of $f_{\rho, n}$ in $U_C$. The use of $U_C$ in place of the ball allows us to derive the upper bound using much simpler arguments because for $z_1, z_2\in U_C$, they are either very close or $|z_1-z_2|\approx |y_1-y_2|$ where $y_i$ is the imaginary part of $z_i$ and is a real number!

Since the upper bound at $\pm 1$ has been proved in Theorem \ref{thm:hole1}, it suffices to assume that $\zeta\neq \pm 1$. 
\subsection{The setup}
Consider the following rescaled version of $f_n$, centered around $\zeta$
\begin{equation}\label{def:gn}
g_{n}(z)=\frac{1}{n^{\rho+1/2}}f_{\rho, n}\left (\zeta+\frac1n \zeta z\right ), z\in U_C.
\end{equation}
 
The proof consists of the following steps. 
\begin{enumerate}
	\item  Construct a Gaussian process $g_{\infty}$ that shall be the limit of $g_{n}$.
	
	\item Show that for all $\ep\ge 0$, there exists $C$ such that
	\begin{equation}\label{eq:nu:inf}
	\P(N_{g_\infty}(U_C)=0)\le \ep.
	\end{equation}

	\item When the random variables $\xi_i$ are iid standard Gaussian, show that on $U_C$,
	\begin{equation}\label{eq:conv:gn}
	g_n  \xrightarrow{w}  g.
	\end{equation}
 	
	\item Show that \begin{equation}\label{eq:uni}
	\E N^{k}_{g_n}(U_C)\to \E N^{k}_{g_\infty}(U_C)
	\end{equation} for general $\xi_i$ (not necessarily Gaussian).
	
	\item Show that this implies \begin{equation}\label{eq:conv:nu}
	\P(N_{g_n}(U_C)=0)\to \P(N_{g_\infty}(U_C)=0).
	\end{equation}
\end{enumerate}
 These steps are carried out in Sections \ref{sec:proof:ginf}, \ref{sec:proof:ginfU}, \ref{sec:proof:gU}, \ref{sec:proof:uni}, \ref{sec:proof:Puni}, respectively.

\subsection{Construct $g_{\infty}$}  \label{sec:proof:ginf}

We have for all $z, w\in U_C$, 
\begin{eqnarray*}
&& \E g_n(z)\overline{g_n(w)} =  \frac{1}{n^{\rho+1}} \sum_{k=0}^{n} |\zeta|^{2k}a_{k, \rho, n}^{2}(1+\frac1n z)^{k}(1+\frac1n \bar w)^{k} \notag\\
 &=& (1+o_n(1)) \frac1n (1+\frac1n z)^{\rho}(1+\frac1n \bar w)^{\rho} \frac{ \partial ^{2\rho}}{\partial z^{\rho}\partial \bar w^{\rho}}\sum_{k=0}^{n}  (1+\frac1n z)^{k}(1+\frac1n \bar w)^{k} \notag\\
 &=& (1+o_n(1))  \frac{ \partial ^{2\rho}}{\partial z^{\rho}\partial \bar w^{\rho}}\frac{(1+\frac1n z)^{n+1}(1+\frac1n \bar w)^{n+1}-1}{n\left ((1+\frac1n z)(1+\frac1n \bar w)-1\right ) }.\notag
\end{eqnarray*}
where we used $|\zeta|=1$.
When $\rho=0$, we have
\begin{eqnarray*}
 \E g_n(z)\overline{g_n(w)}  &\xrightarrow{n\to \infty}& \frac{1}{z+\bar w} (\exp(z+\bar w)-1) = \int_{0}^{1} e^{(z+\bar w)t} dt= F(z+\bar w).\notag
\end{eqnarray*}
where
\begin{equation}\label{def:F}
F(u) := \int_{0}^{1} e^{t u} dt = \frac{e^{u}-1}{u}.
\end{equation}
Similarly, for all $\rho$, 
we have
\begin{eqnarray*}
	\E g_n(z)\overline{g_n(w)}  &\xrightarrow{n\to \infty}& \frac{\partial ^{2\rho}}{\partial z^{\rho}\partial \bar w^{\rho}}  F(z+\bar w) = F^{(2\rho)}(z+\bar w) = \int_{0}^{1} t^{2\rho}e^{t (z+\bar w)} dt.\notag
\end{eqnarray*}
Moreover, since $\zeta\neq \pm 1$, 
\begin{eqnarray*}
	\E g_n(z) {g_n(w)} &=& (1+o(1))\frac{\partial ^{2\rho}}{\partial z^{\rho}\partial w^{\rho}}  \frac{ \zeta^{2n+2}(1+\frac1n z)^{n+1}(1+\frac1n   w)^{n+1}-1}{n\left (\zeta ^2(1+\frac1n z)(1+\frac1n   w)-1\right ) } \xrightarrow{n\to \infty} 0\notag
\end{eqnarray*}
as the denominator blows up (with its derivatives bounded) and the numerator is bounded.
So, it is logical to define the tentative limit $g_{\infty}$ to be
$$g_\infty(z) = \int_{0}^{1}t^{\rho} e^{zt} dB(t)$$
where $B$ is the standard complex Brownian motion $B(t) = \frac{1}{\sqrt 2} (B_1(t)+\sqrt{-1} B_2(t))$ with $B_1, B_2$ being independent standard real Brownian motions. We shall prove in Section \ref{sec:proof:gU} that $g_n$ indeed converges to $g_{\infty}$ when $\zeta\neq \pm 1$.

\subsection{Upper bound the hole radius for $g_{\infty}$}   \label{sec:proof:ginfU}
In this section, we want to show that there exists at least one root of $g_{\infty}$ in $U_C$ with high probability (as $C\to \infty$).
By Chebyshev's inequality, it suffices to show the following
\begin{equation}\label{eq:var:e:inf}
\Var N_{g_\infty}(U_C)=o_{C\to\infty}(\E N_{g_\infty}(U_C))^{2}.
\end{equation}
Indeed, we have
$$\P(N_{g_\infty}(U_C)=0)\le \frac{\Var N_{g_\infty}(U_C)}{(\E N_{g_\infty}(U_C))^{2}}=o(1).$$
So, by choosing $C$ to be sufficiently large, the probability of $g_{\infty}$ having no roots can be arbitrarily small.

To prove \eqref{eq:var:e:inf}, let us evaluate $\E N_{g_\infty}(U_C)$. Since $g_{\infty}$ is a Gaussian analytic function (\cite{HKPV}), we can use Kac-Rice formula for Gaussian case. By \cite[Corollary 3.4.2]{HKPV}, we have
\begin{eqnarray}
\E N_{g_\infty}(U_C) &=& \int_{U_C}\rho_1(z):=\int_{U_C}   \frac{V - T^{2} S^{-1}}{\pi S} dz=\int_{U_C}   \frac{VS - T^{2} }{\pi S^{2}} dz\notag
\end{eqnarray}
 where
    $$S(u) = \E |g_{\infty}(z)|^{2}, 
    \quad T(u) = \E  g_{\infty}'(z) \overline{g_{\infty}(z)}, 
    \quad V(u) = \E |g_{\infty}'(z)|^{2}.$$
Let $u = z+\bar z\in \R$. 
By the definition \eqref{def:F} of $F$, we have
$$S  = F^{(2\rho)}(u).$$
Taking derivative gives
$$T = F^{(2\rho+1)}(u), V =   F^{(2\rho+2)}(u).$$
Since $$F(u) = \frac{e^{u}-1}{u} = \sum_{k=0}^{\infty} \frac{u^{k}}{(k+1)!},$$
we have 
\begin{equation}
F^{(k)}(0) = \frac{1}{k+1}.\nonumber
\end{equation}
 Therefore, 
\begin{equation}\label{eq:prop:F}
F^{(\rho+2)}(0) F^{(\rho)}(0) - \left (F^{(\rho+1)}(0)\right )^{2}\neq 0.
\end{equation}
And so, $V(0)S(0) - T^{2}(0) = \frac{1}{(2\rho+1)(2\rho+3)} - \frac{1}{(2\rho+2)^{2}}>0.$
By choosing $\delta$ to be sufficiently small, it holds that for all $|u|\le 2\delta$, we have
\begin{equation}\label{eq:devF}
V(u)S(u) - T^{2}(u)=\Theta_{\rho}(1) \quad \mbox{and}\quad S(u) = \Theta_{\rho}(1).
\end{equation}
By the definition of $U_C$, $u = z+\bar z\in [-2\delta, 2\delta]$ for all $z\in U_C$.  Hence, $\rho_1(z) = \Theta_{\rho}(1)$ which implies
\begin{eqnarray}
 \E N_{g_\infty}(U_C) = \Theta(C).\notag 
 \end{eqnarray}
It remains to show the following. 
\begin{lemma} \label{lm:var} We have
	 \begin{eqnarray}
	 \Var N_{g_\infty}(U_C) &\ll& C \log C.
	 \end{eqnarray}
\end{lemma}

\begin{proof} [Proof of Lemma \ref{lm:var}]
	Let $N = N_{g_\infty}(U_C)$. We have
	\begin{eqnarray}
	\Var (N_{g_\infty}(U_C)) &=& \E N(N-1) + \E N = \E N(N-1) + O(C).\notag
	\end{eqnarray}
	So, it suffices to show that $\E N(N-1) \ll C\log C$. By \cite[Corollary 3.4.2]{HKPV}, 
	\begin{eqnarray}
	\E N(N-1) &=& \int_{U_C} \int_{U_C} \rho_2(z_1, z_2) - \rho_1(z_1) \rho_1(z_2) dz_1 dz_2
	\end{eqnarray}
	where 
	\begin{eqnarray}\label{eq:rho2}
	\pi^{2}\rho_2(z_1, z_2) =\frac{\text{per} (\mathcal V - \mathcal T \mathcal S^{-1} \mathcal T^{*})}{\det (\mathcal S)},
	\end{eqnarray}
	with $\mathcal S, \mathcal T, \mathcal V$ being $2\times 2$ matrices defined by
 	$$\mathcal S_{i, j} = \E g(z_i)\bar g(z_j) = \int_{0}^{1} t^{2\rho} e^{t u_{ij}}dt = F^{(2\rho)}(u_{ij}), \quad u_{ij} = z_i+\bar z_j$$
 	$$\mathcal T_{ij} = \E g'(z_i)\bar g(z_j) = \int_{0}^{1} t^{2\rho+1} e^{t u_{ij}}dt = F^{(2\rho+1)}(u_{ij})$$
 	$$V_{ij} = \E g'(z_i)\overline{ g'}(z_j) = \int_{0}^{1}  t^{2\rho+2}e^{t u_{ij}}dt = F^{(2\rho+2)}(u_{ij}).$$

 Since $\mathcal T^*=\mathcal T$, by letting  $\mathcal W = \mathcal T \mathcal S^{-1} \mathcal T^{*}$, we have  
\begin{eqnarray}
\mathcal W_{ij} &=&  =\sum_{k=1}^{2} \sum_{\ell = 1}^{2} \mathcal T_{ik} (\mathcal S^{-1})_{k\ell}   \mathcal T_{\ell j} 
\end{eqnarray}
where 
$$\det (\mathcal S) \mathcal S^{-1} = \begin{pmatrix}
\mathcal S_{22} & -\mathcal S_{12}\\
-\mathcal S_{21} & \mathcal S_{11}
\end{pmatrix}.$$

We present a straightforward observation from the forms of $\rho_1$ and $\rho_2$.
\begin{lemma}\label{lm:off:diag}
If $\mathcal S, \mathcal T, \mathcal V, \mathcal W$ were diagonal (namely, setting the off-diagonal entries to $0$), then $\rho_2(z_1, z_2) - \rho_1(z_1)\rho_1(z_2) = 0.$ .
\end{lemma}
	Motivated by this, we will show that for $z_1$ and $z_2$ far away, the above matrices are indeed diagonally dominated, and hence $\rho_2(z_1, z_2) - \rho_1(z_1)\rho_1(z_2)$ is small. In particular, we show the following.
	\begin{lemma} [Off diagonal]
		For all $D_0$ satisfying $10\le D_0\le C$, let  $\mathcal D_{C, D_0} = \{(z_1, z_2)\in U_{C}^{2}: |z_1 - z_2|\ge D_0\}$. For all $(z_1, z_2)\in U_C$, it holds that
		$$|\rho_2(z_1, z_2) - \rho_1(z_1) \rho_1(z_2)| \ll |z_1 - z_2|^{-1}.$$
		In particular, we have
		$$\iint_{\mathcal D_{C, D_0}} \rho_2(z_1, z_2) - \rho_1(z_1) \rho_1(z_2)  dz_1 dz_2\ll C \log(C/D_0).$$
	\end{lemma}

When $z_1$ and $z_2$ are close, we show that $\rho_2(z_1, z_2) - \rho_1(z_1)\rho_1(z_2)$ is bounded and hence the contribution from the diagonal region is negligible.
\begin{lemma}[Diagonal]\label{lm:diag}
	There exists a constant $M$, independent of $C$, such that for all $(z_1, z_2)\in U_{C}^{2}$, we have
	$$|\rho_2(z_1, z_2) - \rho_1(z_1) \rho_1(z_2)| \ll M.$$
	This implies
	$$\iint_{U_{C}^{2}\setminus \mathcal D_{C, D_0}} \left (\rho_2(z_1, z_2) - \rho_1(z_1) \rho_1(z_2) \right ) dz_1 dz_2\ll CD_0.$$
\end{lemma}

Assuming these lemmas, letting $D_0 = \log C$, we get that $\Var N(U_C) \ll C \log C $ as desired. This finishes the proof of Lemma \ref{lm:var}.
\end{proof}

	\begin{proof}[Proof of Lemma \ref{lm:off:diag}]
		 We write $x_i = \Re(z_i)$ and $y_i = \Im(z_i)$ for $i=1, 2$.  Note that $|x_i|\le \delta$ for all $i$. By \eqref{eq:devF},
		 \begin{equation}\label{eq:S:11}
		 |\mathcal S_{11}| = \Theta(1), |\mathcal S_{22}| = \Theta(1).
		 \end{equation}
		 Let $|z_1 - z_2| = D$, we have $D\ge D_0\ge 10$. So,
		 $|y_1 - y_2|\gg D$ and hence 
		 \begin{equation}\label{eq:u}
		 |u_{12}| = |u_{21}|\gg D.
		 \end{equation}
		 For all $u\in \C$, since $(u F(u))^{(k)} = kF^{(k-1)}+ uF^{(k)}$ and since the left-hand side equals $e^{u}$ for all $k$, we get
		 $$F^{(k)}(u) = \frac{e^{u}-k F^{(k-1)}(u)}{u}.$$
		For $u_{12} = z_1+\bar z_2$, we have $|\Re(u_{12})|\le 2\delta$ and so, 
		$$|e^{u_{12}}| = e^{\Re(u_{12})}=O(1).$$
		Hence, by \eqref{eq:u} and induction in $k$, it holds for all $k\le \rho$ that
		$$|F^{(k)}(u_{12})|\ll \frac{1}{|u_{12}|}$$
		 which gives 
		 \begin{equation}\label{eq:S:12}
		 |\mathcal S_{12}| = O(D^{-1}).
		 \end{equation}
		 Similarly, $|\mathcal S_{21}| = O(D^{-1})$.
		 
		 Thus, 
		 $$\det(\mathcal S) = \mathcal S_{11} \mathcal S_{22} - O(D^{-2}) = (1+O(D^{-2})) \mathcal S_{11} \mathcal S_{22} = \Theta(1).$$
		 And so, 
		 $$\mathcal S^{-1} = \begin{pmatrix}
		 \frac{1+O(D^{-2})}{\mathcal S_{11}} & O(D^{-1})\\
		 O(D^{-1})&  \frac{1+O(D^{-2})}{\mathcal S_{22}}
		 \end{pmatrix}.$$
		 Similarly, the same bounds as in \eqref{eq:S:11} and \eqref{eq:S:12} hold for $\mathcal T$ and $\mathcal V$ in place of $\mathcal S$.
		 
		 Using these bounds, we get that
		 $\mathcal W_{12}$ is the sum of 4 terms each of which is of order $O(D^{-1})$. So, $|\mathcal W_{12}| = O(D^{-1})$. Likewise, $|W_{21}| = O(D^{-1})$.
		 
		 Similarly, 
		 $$\mathcal W_{11} = O(D^{-1}) + \mathcal T_{11}^{2} (\mathcal S^{-1})_{11}= O(D^{-1}) + \frac{(1+O(D^{-2}))\mathcal T_{11}^{2}}{\mathcal S_{11}} = \frac{\mathcal T_{11}^{2}}{\mathcal S_{11}} + O(D^{-1}) = \Theta(1)$$
		 and
		 $$\mathcal W_{22} = \frac{\mathcal T_{22}^{2}}{\mathcal S_{22}} + O(D^{-1}) = \Theta(1).$$
		 
		 Therefore, $$|(\mathcal V-\mathcal W)_{12}|=O(D^{-1}), |(\mathcal V-\mathcal W)_{12}|=O(D^{-1}).$$
		 And
		 $$(\mathcal V-\mathcal W)_{11} = \mathcal V_{11} - \frac{\mathcal T_{11}^{2}}{\mathcal S_{11}} + O(D^{-1}) = \Theta(1)$$
		 where in the last equality, we used \eqref{eq:devF}. So, 
		 $$\text{per}(\mathcal V - \mathcal W) = \left (\mathcal V_{11} - \frac{\mathcal T_{11}^{2}}{\mathcal S_{11}}\right )\left (\mathcal V_{22} - \frac{\mathcal T_{22}^{2}}{\mathcal S_{22}}\right ) + O(D^{-1}).$$
		 
		 All in all, we get
		 \begin{eqnarray}
		 \pi^{2}\left |\rho_2(z_1, z_2) - \rho_1(z_1)\rho_1(z_2)\right | &=& \frac{\left (V_{11} - \frac{T_{11}^{2}}{\mathcal S_{11}}\right )\left (V_{22} - \frac{T_{22}^{2}}{\mathcal S_{22}}\right ) + O(D^{-1})}{(1+O(D^{-2})) \mathcal S_{11} \mathcal S_{22} }  - \frac{\left (V_{11} - \frac{T_{11}^{2}}{\mathcal S_{11}}\right )\left (V_{22} - \frac{T_{22}^{2}}{\mathcal S_{22}}\right )}{\mathcal S_{11} \mathcal S_{22} } \notag\\
		 &=& \frac{O(D^{-1})}{\Theta(1)}= O(D^{-1}).\notag
		 \end{eqnarray}
		 Integrating this over $\mathcal D_{C, D_0}$ we get
		 \begin{eqnarray}
		 \iint_{\mathcal D_{C, D_0}} \rho_2(z_1, z_2) - \rho_1(z_1) \rho_1(z_2)  dz_1 dz_2 & \ll & \iint_{\mathcal D_{C, D_0}}  |z_1-z_2|^{-1} dz_1 dz_2\notag\\
		 & \ll & \int_{U_{C}}dz_1\int_{w\in [-2\delta, 2\delta2]\times [-2C, 2C], |w|\ge D_0}  |w|^{-1} dw\notag\\
		 &=& \Theta(C \int_{D_0}^{C} y^{-1}dy) = \Theta(C \log(C/D_0)).\notag
		 \end{eqnarray}
		 This finishes the proof of Lemma \ref{lm:off:diag}.
			\end{proof}

	\begin{proof}[Proof of Lemma \ref{lm:diag}]
		Since $\rho_1(z_1)\rho_1(z_2)$ is bounded over $U_{C}^{2}$, we only need to show the boundedness of $\rho_2$. By the first part of Lemma \ref{lm:off:diag}, we can reduce to the region 
		$$U_{\text{diag}} := \{(z_1, z_2)\in U_C^{2}: |z_1-z_2|\ll 1\}.$$ 
		Since $\rho_2$ can be written as a function of $(u_{ij})_{i, j=1, 2}$, it is also a function of $x_1, x_2$ and $\Delta_y:=y_1 - y_2$ where we recall $x_i = \Re(z_i)$ and $y_i = \Im(z_i)$. Note that $U_{\text{diag}}$ is a subset of $\{(z_1, z_2): |x_1|\le \delta, |x_2|\le \delta, \Delta_y\ll 1\}$ which is a compact set. 
		So, if we can show that $\rho_2$ is in fact a continuous function of $x_1, x_2$ and $\Delta_y:=y_1 - y_2$, we conclude that it is bounded $U_{\text{diag}}$.
To show continuity, note that the only possible singularities of $\rho_2$ occur when $\det (\mathcal S) = 0$. Thus, it suffices to show the following.
		\begin{lemma}\label{lm:equal}
			If $\det(\mathcal S(z_1, z_2)) = 0$ then $z_1 = z_2$.
		\end{lemma}
	
	\begin{lemma}\label{lm:discont} For all $z\in U_C$, $\rho_2(z_1, z_2)$ is continuous at $(z_1, z_2) = (z, z)$.
	\end{lemma}
Proving these lemmas will complete the proof of Lemma \ref{lm:diag}.
	\end{proof}
	\begin{proof}[Proof of Lemma \ref{lm:equal}]
	 Assume that $\det (\mathcal S(z_1, z_2))=0$. Since $\mathcal S(z_1, z_2)$ is a $2\times 2$ (complex) matrix, there exist deterministic complex numbers $w_1, w_2$ such that 
	 $$(w_1 \ w_2)\mathcal S (w_1 \ w_2)^{\mathcal T}=0.$$
	 In other words, 
	 $$\E |w_1 g(z_1) + w_2 g(z_2)|^{2}=0.$$
	 Since the left-hand side equals $\int_{0}^{1} |w_1 e^{tz_1}+w_2 e^{tz_2}|^{2}dt$, we conclude that the integrand is 0 for almost all $t$ (and hence for all $t$ by continuity). Therefore, it is necessary that $z_1 = z_2$.
	\end{proof}
	
	 	\begin{proof}[Proof of Lemma \ref{lm:discont}] 
	 We need to show that for all $z\in \C$ (or jut $U_C$ if needed), 
	 $$\lim_{(\ep, \delta)\to(0, 0)} \rho_2(z, z+\ep+i\delta) \text{ exists}.$$
	 We shall perform Taylor expansion to the order 2 of the functions appearing in \eqref{eq:rho2}. Here $z_1 = z = x+iy, z_2 = z+\ep+i\delta$. Then
	 $$u_{11} = 2x=: u, u_{12} = (2x+\ep)-i\delta, u_{22} = 2x+2\ep.$$
	 Let 
	 $$a=  F^{(2\rho)}(u), b=F^{(2\rho+1)}(u), c=F^{(2\rho+2)}(u), d= F^{(2\rho+3)}(u), e=F^{(2\rho+4)}(u).$$

	 So, 
	 $$\mathcal S_{11} = F^{(2\rho)}(u)=a,$$
	 $$\mathcal S_{22} = F^{(2\rho)}(u+2\ep)\approx a+2\ep b+2\ep^{2}c,$$
	 $$\mathcal S_{12}=\overline{ \mathcal S_{21}} = F^{(2\rho)}(u+\ep-i\delta) \approx a + (\ep-i\delta) b + \frac{(\ep-i\delta)^{2}}{2}c = \left (a + \ep b+\frac{\ep^{2}-\del^{2}}{2}c\right )   -i\left (\delta b + \ep\del c\right ).$$
	 That is
	\begin{equation}\label{eq:S:d}
	\mathcal S \approx \begin{pmatrix}
	a & \left (a + \ep b+\frac{\ep^{2}-\del^{2}}{2}c\right )   -i\left (\delta b + \ep\del c\right )\\
	\left (a + \ep b+\frac{\ep^{2}-\del^{2}}{2}c\right )   +i\left (\delta b + \ep\del c\right ) & a+2\ep b+2\ep^{2}c
	\end{pmatrix}.
	\end{equation}
	 
	 The denominator of $\rho_2$ is
	\begin{eqnarray}
	\det(\mathcal S) &=& \mathcal S_{11} \mathcal S_{22} - \mathcal S_{12} \mathcal S_{21} \approx a( a+2\ep b+2\ep^{2}c) - \left (a + \ep b+\frac{\ep^{2}-\del^{2}}{2}c\right )^{2}- \left (\delta b + \ep\del c\right )^{2}\notag\\
	&=& 2\ep^{2} ac-\ep^{2}b^{2}-(\ep^{2}-\delta^{2})ac-\del^{2}b^{2}+o(\ep^{2}+\del^{2})\approx (\ep^{2}+\delta^{2})(ac-b^{2}).\notag
	\end{eqnarray}
	Note that by \eqref{eq:devF}, $ac-b^{2}\neq 0$.
	
	Since $\mathcal T$ and $\mathcal V$ are similar to $\mathcal S$, we get
	$$\mathcal T \approx \begin{pmatrix}
	 b & \left (b + \ep c+\frac{\ep^{2}-\del^{2}}{2}d\right )   -i\left (\delta c + \ep\del d\right )\\
	 \left (b + \ep c+\frac{\ep^{2}-\del^{2}}{2}d\right )   +i\left (\delta c + \ep\del d\right ) & b+2\ep c+2\ep^{2}d
	\end{pmatrix}$$
	and 
		$$\mathcal V \approx \begin{pmatrix}
	c & \left (c + \ep d+\frac{\ep^{2}-\del^{2}}{2}e\right )   -i\left (\delta d + \ep\del e\right )\\
	\left (c + \ep d+\frac{\ep^{2}-\del^{2}}{2}e\right )   +i\left (\delta d + \ep\del e\right ) & c+2\ep d+2\ep^{2}e
	\end{pmatrix}.$$
	
	Note that the numerator of \eqref{eq:rho2} for $\rho_2$ equals
	$$\per (\mathcal V-\mathcal W) = \frac{1}{\det(\mathcal S)^{2}} \per ((\det \mathcal S) \mathcal V  - (\det \mathcal S) \mathcal W) .$$
	So far, \eqref{eq:rho2} becomes
	$$\pi^{2}\rho_2 \approx \frac{\per ((\det \mathcal S) \mathcal V  - (\det \mathcal S) \mathcal W)}{ (\ep^{2}+\delta^{2})^{3}(ac-b^{2})^{3}}=: \frac{\per (\mathcal Y)}{(\ep^{2}+\delta^{2})^{3}(ac-b^{2})^{3}}.$$
	
	Next, we write down $(\det \mathcal S) \mathcal W$. We have
$$(\det \mathcal S) \mathcal W_{11} =  \mathcal T_{11} \mathcal S_{22} \mathcal T_{11} - \mathcal T_{12} \mathcal S_{21} \mathcal T_{11}- \mathcal T_{11}\mathcal S_{12}\mathcal T_{21}+\mathcal T_{12}\mathcal S_{11}\mathcal T_{21}$$
$$(\det \mathcal S) \mathcal W_{12} =  \mathcal T_{11} \mathcal S_{22} \mathcal T_{12} - \mathcal T_{12} \mathcal S_{21} \mathcal T_{12}- \mathcal T_{11}\mathcal S_{12}\mathcal T_{22}+\mathcal T_{12}\mathcal S_{11}\mathcal T_{22}$$
$$(\det \mathcal S) \mathcal W_{21} =  \mathcal T_{21} \mathcal S_{22} \mathcal T_{11} - \mathcal T_{22} \mathcal S_{21} \mathcal T_{11}- \mathcal T_{21}\mathcal S_{12}\mathcal T_{21}+\mathcal T_{22}\mathcal S_{11}\mathcal T_{21}$$
$$(\det \mathcal S) \mathcal W_{22} =  \mathcal T_{21} \mathcal S_{22} \mathcal T_{12} - \mathcal T_{22} \mathcal S_{21} \mathcal T_{12}- \mathcal T_{21}\mathcal S_{12}\mathcal T_{22}+\mathcal T_{22}\mathcal S_{11}\mathcal T_{22}$$
	So, 
\begin{eqnarray}
\mathcal Y_{11} &=& (\ep^{2}+\delta^{2})(ac-b^{2}) c - (\mathcal T_{11} \mathcal S_{22} \mathcal T_{11} - \mathcal T_{12} \mathcal S_{21} \mathcal T_{11}- \mathcal T_{11}\mathcal S_{12}\mathcal T_{21}+\mathcal T_{12}\mathcal S_{11}\mathcal T_{21})\notag\\
&=& {(\ep^{2}+\delta^{2})(ac-b^{2}) c} -   b^{2} (a+2\ep b+2\ep^{2}c) \notag\\
&&{+ 2b\Re\left ( \left (b + \ep c+\frac{\ep^{2}-\del^{2}}{2}d\right ) + i\left (\delta c + \ep\del d\right )\right ) \left (\left (a + \ep b+\frac{\ep^{2}-\del^{2}}{2}c\right )   -i\left (\delta b + \ep\del c\right )\right )} \notag\\
&&- a\left (b + \ep c+\frac{\ep^{2}-\del^{2}}{2}d\right )^{2}   - a \left (\delta c + \ep\del d\right )^{2} \notag\\
&=& {(\ep^{2}+\delta^{2})(ac-b^{2}) c} -   b^{2} (a+2\ep b+2\ep^{2}c) \notag\\
&&{+ 2b \left ( \left (b + \ep c+\frac{\ep^{2}-\del^{2}}{2}d \right )  \left (a + \ep b+\frac{\ep^{2}-\del^{2}}{2}c\right )  +  \left (\delta c + \ep\del d\right ) \left (\delta b + \ep\del c\right )\right )}  \notag\\
&&- a\left (b + \ep c+\frac{\ep^{2}-\del^{2}}{2}d\right )^{2}   - a \left (\delta c + \ep\del d\right )^{2} \notag.
\end{eqnarray}
Grouping the terms with $\ep$, $\delta$, $\ep \delta$, $\ep^{2}$, $\delta^{2}$ and smaller order terms together, we get
\begin{eqnarray}
\mathcal Y_{11} &=& \ep\left (-2b^{3}+2b^{3}+2abc-2abc\right ) + \ep^{2}\left ((ac-b^{2}) c + b^{2}c+abd-ac^{2}-abd\right ) \notag\\
&& + \del^{2}\left ((ac-b^{2}) c -b^{2}c-abd+2b^{2}c+abd-ac^{2}\right )+ O(\ep^{3}+\ep\delta^{2}+\ep^{2}\delta+\del^{3})\notag\\
&=&  O(\ep^{3}+\ep\delta^{2}+\ep^{2}\delta+\del^{3})\notag.
\end{eqnarray}

We now try to accomplish the same estimate for the other three entries of $\mathcal Y$. We have
\begin{eqnarray}
\mathcal Y_{22} &=& (\ep^{2}+\delta^{2})(ac-b^{2}) \left(  c+2\ep d+2\ep^{2}e\right ) - (\mathcal T_{22}\mathcal S_{11}\mathcal T_{22} - \mathcal T_{22} \mathcal S_{21} \mathcal T_{12}- \mathcal T_{21}\mathcal S_{12}\mathcal T_{22}+\mathcal T_{21} \mathcal S_{22} \mathcal T_{12} )\notag\\
&=& { (\ep^{2}+\delta^{2})(ac-b^{2})  \left(  c+2\ep d+2\ep^{2}e\right )} -a\left (b+2\ep c+2\ep^{2}d\right )^{2}  \notag\\
&&{+ 2\left (b+2\ep c+2\ep^{2}d\right )\Re\left( \left ( \left (b + \ep c+\frac{\ep^{2}-\del^{2}}{2}d\right ) + i\left (\delta c + \ep\del d\right )\right )\right .}\notag\\
&&{\hspace{43mm} \left .\left (\left (a + \ep b+\frac{\ep^{2}-\del^{2}}{2}c\right )   -i\left (\delta b + \ep\del c\right )\right )\right ] }\notag\\
&&- \left (b+2\ep c+2\ep^{2}d\right )\left (b + \ep c+\frac{\ep^{2}-\del^{2}}{2}d\right )^{2}   - \left (b+2\ep c+2\ep^{2}d\right ) \left (\delta c + \ep\del d\right )^{2} \notag.
\end{eqnarray}
So, 
\begin{eqnarray}
\mathcal Y_{22} 
&=& {(\ep^{2}+\delta^{2})(ac-b^{2})  \left(  c+2\ep d+2\ep^{2}e\right )} -a\left (b+2\ep c+2\ep^{2}d\right )^{2}  \notag\\
&&{+ 2\left (b+2\ep c+2\ep^{2}d\right ) \left ( \left (b + \ep c+\frac{\ep^{2}-\del^{2}}{2}d \right )  \left (a + \ep b+\frac{\ep^{2}-\del^{2}}{2}c\right )  +  \del^{2}\left (  c + \ep  d\right ) \left (  b + \ep  c\right )\right ) } \notag\\
&&- {\left (a+2\ep b+2\ep^{2}c\right )\left (b + \ep c+\frac{\ep^{2}-\del^{2}}{2}d\right )^{2}   - \left (a+2\ep b+2\ep^{2}c\right ) \del^{2}\left (c + \ep d\right )^{2}} \notag
\end{eqnarray}
Comparing this with $\mathcal Y_{11}$, we get
\begin{eqnarray}
\mathcal Y_{22}  
&=&  \mathcal Y_{11} + {(\ep^{2}+\delta^{2})(ac-b^{2})  \left( 2\ep d+2\ep^{2}e\right ) }+ (2\ep b^{3} +2\ep^{2} b^{2}c - 4 \ep^{2} ac^{2} - 4\ep^{2}abd - 4\ep abc)\notag\\
&&{+ 2\left (2\ep c+2\ep^{2}d\right ) \left ( \left (b + \ep c+\frac{\ep^{2}-\del^{2}}{2}d \right )  \left (a + \ep b+\frac{\ep^{2}-\del^{2}}{2}c\right )  +  \left (\delta c + \ep\del d\right ) \left (\delta b + \ep\del c\right )\right )}  \notag\\
&&- {\left (2\ep b+2\ep^{2}c\right )\left (b + \ep c+\frac{\ep^{2}-\del^{2}}{2}d\right )^{2}   - \left (2\ep b+2\ep^{2}c\right ) \del^{2}\left (c + \ep d\right )^{2}} +O(\ep^{3}+\ep\delta^{2}+\ep^{2}\delta+\del^{3})\notag\\
&=&  O(\ep^{3}+\ep\delta^{2}+\ep^{2}\delta+\del^{3})  + (2\ep b^{3} +2\ep^{2} b^{2}c - 4 \ep^{2} ac^{2} - 4\ep^{2}abd - 4\ep abc)\notag\\
&&{+ 2\left (2\ep c+2\ep^{2}d\right )  \left (b + \ep c  \right )  \left (a + \ep b \right )}  - {\left (2\ep b+2\ep^{2}c\right )\left (b + \ep c \right )^{2}   }  \notag
\end{eqnarray}
which gives
\begin{eqnarray}
\mathcal Y_{22}  
&=&O(\ep^{3}+\ep\delta^{2}+\ep^{2}\delta+\del^{3})  + (2\ep b^{3} +2\ep^{2} b^{2}c - 4 \ep^{2} ac^{2} - 4\ep^{2}abd - 4\ep abc)\notag\\
&&{+ 2\left (2\ep c+2\ep^{2}d\right )  \left (ab + \ep b^{2}+\ep ac \right )}  - {\left (2\ep b+2\ep^{2}c\right )\left (b^{2} +2 \ep bc \right )}  \notag\\
&=&O(\ep^{3}+\ep\delta^{2}+\ep^{2}\delta+\del^{3})  + (2\ep b^{3} +2\ep^{2} b^{2}c - 4 \ep^{2} ac^{2} - 4\ep^{2}abd - 4\ep abc)\notag\\
&&{+  \left (4\ep abc + 4\ep^{2}  b^{2}c+4\ep^{2} ac^{2} + 4\ep^{2} abd\right )}  - { \left (2\ep b^{3} +6 \ep^{2} b^{2}c  \right )}  \notag\\
&=&O(\ep^{3}+\ep\delta^{2}+\ep^{2}\delta+\del^{3}) \notag.
\end{eqnarray}
Finally, 
\begin{eqnarray}
\mathcal Y_{12} &=& (\ep^{2}+\delta^{2})(ac-b^{2})  \left (\left (c + \ep d+\frac{\ep^{2}-\del^{2}}{2}e\right )   -i\left (\delta d + \ep\del e\right )\right )\notag\\
&& - (T_{11} \mathcal S_{22} T_{12} - T_{12} \mathcal S_{21} T_{12}- T_{11}\mathcal S_{12}T_{22}+T_{12}\mathcal S_{11}T_{22})\notag\\
&=& O(\ep^{3}+\ep\delta^{2}+\ep^{2}\delta+\del^{3}) + (\ep^{2}+\delta^{2})(ac-b^{2}) c \notag\\
&&{- b \left (a+2\ep b+2\ep^{2}c\right ) \left (\left (b + \ep c+\frac{\ep^{2}-\del^{2}}{2}d\right )   -i\left (\delta c + \ep\del d\right )\right ) }\notag\\
&& {+ \left (\left (b + \ep c+\frac{\ep^{2}-\del^{2}}{2}d\right )   -i\left (\delta c + \ep\del d\right )\right )^{2} \left (\left (a + \ep b+\frac{\ep^{2}-\del^{2}}{2}c\right )   +i\left (\delta b + \ep\del c\right ) \right )}\notag\\
&& {+ b\left (b+2\ep c+2\ep^{2}d\right ) \left (\left (a + \ep b+\frac{\ep^{2}-\del^{2}}{2}c\right )   -i\left (\delta b + \ep\del c\right )\right )}  \notag\\
&&- a \left ( b+2\ep c+2\ep^{2}d\right ) \left ( \left (b + \ep c+\frac{\ep^{2}-\del^{2}}{2}d\right )   -i\left (\delta c + \ep\del d\right )\right ).\notag
\end{eqnarray}
So, 
\begin{eqnarray}
\mathcal Y_{12} &=& O(\ep^{3}+\ep\delta^{2}+\ep^{2}\delta+\del^{3}) + (\ep^{2}+\delta^{2})(ac-b^{2}) c \notag\\
&&{- ab\left ( b + \ep c+\frac{\ep^{2}-\del^{2}}{2}d -i\left (\delta c + \ep\del d\right )\right ) -2\ep b^{2} (b + \ep c -i \delta c)-2\ep^{2}b^{2}c}\notag\\
&& {+b\left ( b + \ep c+\frac{\ep^{2}-\del^{2}}{2}d  -i\left (\delta c + \ep\del d\right )\right ) \left ( a + \ep b+\frac{\ep^{2}-\del^{2}}{2}c   +i\left (\delta b + \ep\del c\right ) \right )}\notag\\
&&  {+ \ep c \left (\left (b+\ep c \right )   -i\delta c \right )  \left ( a + \ep b   +i \delta b   \right ) + \frac{\ep^{2}-\del^{2}}{2}d ab}\notag\\
	&&  {- i\delta c \left (\left (b+\ep c \right )   -i\delta c \right )  \left ( a + \ep b   +i \delta b   \right ) -i\ep\delta abd}\notag\\
&& { +b^{2}\left ( a + \ep b+\frac{\ep^{2}-\del^{2}}{2}c   -i\left (\delta b + \ep\del c\right )\right )+  2\ep bc \left ( a + \ep b -i \delta b  \right )+2\ep^{2}abd  }\notag\\
&&- ab \left (  b + \ep c+\frac{\ep^{2}-\del^{2}}{2}d  -i\left (\delta c + \ep\del d\right )\right ) -   2\ep a c\left (  b + \ep c   -i \delta c \right ) +2\ep^{2}abd\notag
\end{eqnarray}
giving
\begin{eqnarray}
\mathcal Y_{12} &=& O(\ep^{3}+\ep\delta^{2}+\ep^{2}\delta+\del^{3}) + (\ep^{2}+\delta^{2})(ac-b^{2}) c \notag\\
&&{- ab\left ( b + \ep c+\frac{\ep^{2}-\del^{2}}{2}d -i\left (\delta c + \ep\del d\right )\right ) -2\ep b^{2} (b + \ep c -i \delta c)-2\ep^{2}b^{2}c}\notag\\
&& {+b^{2}\left (a + \ep b+\frac{\ep^{2}-\del^{2}}{2}c   +i\left (\delta b + \ep\del c\right ) \right ) + \ep bc\left (a + \ep b +i \delta b  \right )+\frac{\ep^{2}-\del^{2}}{2}abd  }\notag\\
&&{ -i \delta bc\left (a + \ep b +i \delta b  \right ) -i \ep\del abd}\notag\\
&&  {+   \ep bc\left ( a + \ep b   +i \delta b   \right )+\ep^{2} c^{2}\left ( a + \ep b   +i \delta b   \right )  -i\ep \delta a c^{2} + \frac{\ep^{2}-\del^{2}}{2}d ab}\notag\\
&&{ - i\delta bc \left ( a + \ep b   +i \delta b   \right ) -i\ep \delta a c^{2}  -\delta^{2} a c^{2}    -i\ep\delta abd}\notag\\
&& { +b^{2}\left ( a + \ep b+\frac{\ep^{2}-\del^{2}}{2}c   -i\left (\delta b + \ep\del c\right )\right )+  2\ep bc \left ( a + \ep b -i \delta b  \right )+2\ep^{2}abd  }\notag\\
&&- ab \left (  b + \ep c+\frac{\ep^{2}-\del^{2}}{2}d  -i\left (\delta c + \ep\del d\right )\right ) -   2\ep a c\left (  b + \ep c   -i \delta c \right ) +2\ep^{2}abd\notag\\
&=& O(\ep^{3}+\ep\delta^{2}+\ep^{2}\delta+\del^{3}). \notag 
\end{eqnarray}
Since the product of any two terms in $\{\ep^{3}, \ep\delta^{2}, \ep^{2}\delta, \del^{3}\}$ is bounded by $(\ep^{2}+\delta^{2})^{3}$, we yield the continuity of $\rho_2$.
	 \end{proof}

\subsection{More on $N_{g_n} (U_C)$  and $N_{g_\infty} (U_C)$} Before moving on the next section to show the convergence of $g_n$ to $g_{\infty}$ and their number of roots, we will first show that $N_{g_n} (U_C)$ have uniformly bounded higher moments.
\begin{lemma}\label{lemma:moments:infty} There exists a constant $A$ such that the following holds. For any $\ell\ge 0$, we have
$$\E N_{g_n}^\ell(U_C) \le (A \ell)^{\ell}.$$
\end{lemma}

 \begin{proof}   Let $k\ge A \ell$ for some large constant $A$. We want to bound the probability that $N_{g_n}(U_C)\ge k$. We divide $U_C$ into $O(C\eta^{-2})$ (possibly overlapping) open balls $B_i=B(c_i,\eta)$ centered at $c_i$ of radius $\eta$, which is chosen to be sufficiently small. Then there exists $i$ such that $B_i$ contains at least $s = k \eta^{2}/C$ roots. Then by Hermite interpolation, as $g_n$ is analytic with probability one, we have 
	\begin{equation}\label{eqn:Hermite}
	|g_n(c_i)| \le \frac{1}{s!} \eta^{s} \sup_{z\in B(c_i, \eta)}|g_n^{(s)}(z)|.
	\end{equation}
	By Taylor expanding $g_n^{s}(z)$ around $c_i$, we obtain for any $m\ge 0$ (we later choose $m =\log n$),
	\begin{eqnarray}
	|g_n^{(s)}(z)|\le \sum_{j=s}^{s+m-1} \frac{|g_n^{(j)}(c_i)|}{(j-s)!} \eta^{j-s} +  \sup_{w\in B(c_i, \eta)}\frac{|g_n^{(s+m)}(w)|}{m!} \eta^{m}.\label{eq:double:T}
	\end{eqnarray}
	For each $j$, $g_n^{(j)}(c_i)$ is a Gaussian random variable with mean 0 and variance equals that of $\frac{1}{n^{j+\rho+1/2}} f_n^{(j)}(\zeta+\frac1n \zeta c_i)$, which is of order
	$$(1+o(1))\frac{1}{n^{2j+2\rho+1}} \sum_{h=j}^{n} h^{2}(h-1)^{2}\dots(h-j+1)^{2}a_{h, \rho, n}^{2} |\zeta+\frac1n \zeta c_i|^{2h-2j} = O_{C}(1)$$
	where we used $|\zeta+\frac1n \zeta c_i|^{j}\le (1+2C/n)^{n}\le e^{2C} = O_{C}(1)$. \\
	So, by Gaussianity, for all $M_j\ge 1$, 
	\begin{eqnarray}
	\P(|g_n^{(j)}(c_i)|\gg M_j)\le e^{-M_j}.\notag
	\end{eqnarray}
	Finally, for the supremum term, we observe 
		\begin{eqnarray*}
		\sup_{w\in B(c_i, \eta)}|g_n^{(s+m)}(w)| &\ll& \frac{1}{n^{s+m+\rho+1/2}}\sum_{h=s+m}^{n} h (h-1) \dots (h-s-m+1) |a_{h, \rho, n}| |\xi_h| (|c_i|+\eta)^{h}\\
		&\ll& \frac{1}{n^{1/2}}\sum_{h=1}^{n}  |\xi_h|.
	\end{eqnarray*}
Note that if we hadn't used another round of Taylor expansion in \eqref{eq:double:T} and just applied the above bound to $|g_n^{(s)}(z)|$ and take supremum, the term $n^{-1/2}\sum_{h=1}^{n}  |\xi_h|$, which can be as large as $\sqrt n$, would be too big to handle. Here, we performed \eqref{eq:double:T} so that the extra term $\eta^{m}/m!$ would swallow the $n^{-1/2}\sum_{h=1}^{n}  |\xi_h|$. Indeed, for an $M_0$ to be chosen,
$$\P(\sum_{h=1}^{n}  |\xi_h|\ge M_0 n)\le n\P(|\xi_h|\ge M_0)\le ne^{-M_0}.$$
Thus, 
$$\P(\sup_{w\in B(c_i, \eta)}|g_n^{(s+m)}(w)|\gg n^{1/2} M_0)\ll ne^{-M_0}.$$
Combining all of these events, we conclude that with probability at least $1 - ne^{-M_0} - \sum_{j=s}^{s+m-1} e^{-M_j}$, we have
$$\sup_{z\in B(c_i, \eta)}|g_n^{(s)}(z)|\ll \sum_{j=s}^{s+m-1} \frac{M_j}{(j-s)!} \eta^{j-s} +\frac{n^{1/2} M_0\eta^{m}}{m!}.$$
On this event, 
$$|g_n(c_i)| \ll  \frac{1}{s!} \eta^{s} \left (\sum_{j=s}^{s+m-1} \frac{M_j}{(j-s)!} \eta^{j-s} +\frac{n^{1/2} M_0\eta^{m}}{m!}\right )$$
which only happens with probability at most
$$O_{C}(1)\frac{1}{s!} \eta^{s} \left (\sum_{j=s}^{s+m-1} \frac{M_j}{(j-s)!} \eta^{j-s} +\frac{n^{1/2} M_0\eta^{m}}{m!}\right )$$
since $g_n(c_i)$ is a Gaussian random variable with variance $\Theta_{C}(1)$. All together, we get that the probability that $N_{g_n}^\ell(U_C)\ge k$ is at most (up to a constant depending on $C$),
$$\frac{C}{\eta^{2}}\left [\frac{1}{s!} \eta^{s} \left (\sum_{j=s}^{s+m-1} \frac{M_j}{(j-s)!} \eta^{j-s} +\frac{n^{1/2} M_0\eta^{m}}{m!}\right ) +\sum_{j=s}^{s+m-1} e^{-M_j}\right ]+ ne^{-M_0}$$
for any choice of $\eta$, $M_0, M_j$, with $s = k\eta^{2}/C$. For instance, we choose $s = 8\ell$, we get $\eta =  \frac{\sqrt {8 C\ell}}{\sqrt k}\le \sqrt{\frac{8C}{A}}$. By setting 
$$M_j = s\log \frac1\eta+\eta^{-(j-s)/2}, \quad M_0 = 2\ell \log k + \log n,$$
we obtain the tail probability of 
$$\frac{C}{\eta^{2}}\left [\frac{1}{s!} \eta^{s} \left (e^{\sqrt \eta} +\frac{n^{1/2} M_0\eta^{m}}{m!}\right ) +\eta^{s} \right ]+k^{-2\ell}.$$
Sending $m\to\infty$, the term with $m$ goes to $0$, so we end up with 
$$\eta^{s-2}+k^{-2\ell}\ll (C\ell)^{4\ell} k^{-4\ell+1}+k^{-2\ell}.$$
So, 
\begin{eqnarray}
\E N_{g_n}^{\ell}(U_C) &\ll& (A\ell)^{\ell} + \ell \sum_{k=A\ell}^{\infty} k^{\ell - 1} \P(N_{g_n}(U_C)\ge k)\notag\\
&\ll& (A\ell)^{\ell} +   \ell \sum_{k=A\ell}^{\infty} \left ((C\ell)^{4\ell} k^{-3\ell}+k^{-\ell-1}\right )\ll  (A\ell)^{\ell}\notag
\end{eqnarray}
as desired.

\end{proof}

\subsection{Convergence of $g_n$ to $g_{\infty}$ when $\xi_i$ are iid $\mathcal N(0, 1)$} \label{sec:proof:gU} Now, we prove \eqref{eq:conv:gn}. 
We first start with two simple results for the Gaussian models. 

\begin{lemma} With probability one, $g_n$ and $g_{\infty}$ do not have double roots in $U_C$.  
\end{lemma}

\begin{proof} For $g_n$, if it has a double root then $f_n$ also has a double root. As this is a polynomial of degree $n$, if $f_n(z)$ and $f_n'(z)$ have common roots then the resultant must have zero determinant. But the resultant is a non-degenerate multivariate function of the Gaussians, so it is zero with probability zero. 

For $g_\infty(z)$, for any $\alpha>0$, we divide $U_C$ into $O(C\alpha^{-2})$ balls $B_i$ of radius $\alpha$. We will show that the probability there exists $i$ such that $N_i$, the number of zeros in $B_i$, is greater than $2$ is of order $O(\alpha^4)$, from which we see that the given probability will be bounded by $O(\alpha^2)$ after taking union bound. Indeed, using the boundedness of $\rho_2$ in Lemma \ref{lm:diag},
$$\P(N_i \ge 2) \le \E (N_i (N_i-1)) = \int_{B_i \times B_i} \rho_2(z_1,z_2) dz_1 dz_2 \le O( |B_i| \times |B_i|) = O(\alpha^4).$$ 
Sending $\alpha$ to $0$, we conclude that the probability that $g_{\infty}$ has double roots is 0.
\end{proof}

Our next simple result is the following.

\begin{claim}
With probability one, $g_n(z)$ and $g_{\infty}$ do not have roots on the boundary $\partial U_C$ of $U_C$. 
\end{claim}
\begin{proof} We will show for $g_\infty$ as the treatment for $g_n$ is similar. From \eqref{eq:devF}, we saw that for all $\alpha>0$ sufficiently small, $\rho_1(z) = O(1)$ for all $z \in U_C+B(0,\alpha)$. Let $N$ be the number of roots in $\partial U_C+B(0,\alpha)$, then 
$$\P(N \ge 1) \le \E N = \int_{\partial U_C+B(0,\alpha)} \rho_1(z)dz \le O_C(\alpha).$$
Sending $\alpha$ to $0$, we obtain the claim.
\end{proof}

Our treatment below is similar to \cite[Section 4]{IKM} where instead of real roots, we consider complex roots. First, let $\CH$ be the set of all analytic function on the entire complex plane. We endow $\CH$ with the topology of uniform convergence on the compact sets, which can be generated by the complete separable metric 
$$d(f,g) = \sum_k \frac{1}{2^k} \frac{\|f-g\|_{\bar{D}_k}}{1+ \|f-g\|_{\bar{D}_k}},$$
where $\bar{D}_k = \{z\in \C: |z| \le k\}$ and $\|f\|_K = \sup_{z\in K}|f(z)|$.
  
\begin{lemma}\label{lemma:Hurwitz}  Let $A_C$ be the set of all $f\in \CH$ which do not have multiple roots in $U_C$ and do not have roots over the boundary of $U_C$. Then the set $A_C$ is open.
\end{lemma}

\begin{proof} This follows from Hurwitz's theorem. Indeed, consider a sequence $(f_n)_{n\in \N}$ in $\CH$, which converges to some $f\in A_C$ locally uniformly. We will show that $f_n\in A_C$ for sufficiently large $n$. Let $R>0$ be large such that $U_C \subset D_R =\{z: |z| <R\}$. Let $z_1,\dots, z_d$ be the collections of all zeros of $f$ in $D_R$ with multiplicities $m_1,\dots, m_d$. Let $\alpha>0$ be  sufficiently small such that the open disks $z_i + D_\alpha$ are disjoint, and do not intersect the boundary of the open sets $D_R$ and of $U_C$, except when $z_i$ are on one of these boundaries. By Hurwitz's theorem for sequence of (locally convergent) analytic functions, there exists $n_0$ such that for all $n\ge n_0$, $f_n$ has exactly $m_k$ zeros in $z_k + D_\alpha$. Now if $z_i \in U_C$, then as $f\in A_C$, we must have $m_i=1$, and $f_n$ has exactly one zero in $z_i + D_\alpha$. Thus, $f_n\in A_C$ for all $n\ge n_0$. 
\end{proof}

\begin{lemma}\label{lemma:cont-map} The mapping $f \to \mathcal Z_{U_C}(f)=\{z \in U_C: f(z)=0\}$ to the space of locally finite point measures 
 on $U_C$ endowed with the vague topology is continuous on $A_C$.
\end{lemma}

\begin{proof} This also follows from Hurwitz's theorem with the same argument as in the proof of Lemma \ref{lemma:Hurwitz}, by letting the radius $\alpha$ tend to zero.
\end{proof}

\begin{lemma} We have the following weak convergence (of random elements with values in the metric space $\CH$) 
$$g_n \xrightarrow{w} g_\infty.$$
\end{lemma}

\begin{proof}  By Prokhorov's theorem, it suffices to verify convergence in finite dimensional and tightness. Let $z_1,\dots, z_k$ be complex numbers. We first observe that the convergence in distribution of the Gaussian vector $(g_n(z_1),\dots, g_n(z_k))$ to the Gaussian vector $(g_\infty(z_1),\dots, g_\infty(z_k))$ already follows from our previous computations verifying the convergences of $\E g_n(z_i) \overline{g_n(z_j)}$ and $\E g_n(z_i) g_n(z_j)$ to $\E g_\infty(z_i) \overline{g_\infty(z_j)}$ and $\E g_\infty(z_i) g_\infty(z_j)$, respectively.

We need to verify tightness, for this, it suffices to show that for any $R>0$, there exists $C_R < \infty$ such that
$$\sup_{n} \sup_{|z|\le R} \E |g_n(z)|^2 < C_R.$$
However, this is clear as $$\E |g_n(z)|^2 =  (1+o_n(1))  \frac{ \partial ^{2\rho}}{\partial z^{\rho}\partial \bar w^{\rho}}\frac{(1+\frac1n z)^{n+1}(1+\frac1n \bar w)^{n+1}-1}{n\left ((1+\frac1n z)(1+\frac1n \bar w)-1\right ) }\Big\vert_{w=z}.$$
\end{proof}

\begin{theorem}\label{theorem:distribution-moments} 
We have that $N_{g_n}(U_C) \to N_{g_\infty}(U_C)$ in distribution and for each $k\in \N$, $\lim_{n\to \infty} \E N_{g_n}^k(U_C) = \E N^k_{g_\infty}(U_C)$.
\end{theorem}

\begin{proof} We have that $g_n \to g_\infty$ weakly, they are analytic and with probability one, they all belong to $A_C$. By Lemma \ref{lemma:cont-map}, the point process $\mathcal Z_{U_C}(g_n)$ converges to $\mathcal Z_{U_C}(g_\infty)$ weakly, and hence the number of zeros $N_{g_n}(U_C)$ converges to $N_{g_\infty}(U_C)$ in distribution. 
	In particular, for all $m\in \N$, $p_{n, m}:=\P(N_{g_n}(U_C) = m)\to \P(N_{g_\infty}(U_C) = m):= p_m$ as $n\to\infty$.
	By Fatou's lemma and Lemma \ref{lemma:moments:infty}, it holds for all $\ell\in \N$ that 
	\begin{equation}\label{eq:moment}
	\E N_{g_\infty}^\ell(U_C) \le \liminf_{n} \E N_{g_n}^\ell(U_C) \le (A \ell)^{\ell}.
	\end{equation}
	Fix $k\in \N$, we have for a large constant $M$,
	\begin{eqnarray*}
	&&\left |\E N_{g_n}^k(U_C) - \E N^k_{g_\infty}(U_C)\right |\\
	&\le&  \sum_{m=0}^{M-1} m^{k} |p_{nm} - p_m| + \E N_{g_n}^k(U_C) \textbf{1}_{N_{g_n} (U_C)\ge M} + \E N_{g_\infty} (U_C) \textbf{1}_{N_{g_\infty}^k(U_C)\ge M} \\
	&\le & \sum_{m=0}^{M-1} m^{k} |p_{nm} - p_m| + 2\sup_{\hat n}\left (\E N_{g_{\hat n}}^{2k}(U_C)\right )^{1/2} \P\left (N_{g_{\hat n}} (U_C)\ge M\right )^{1/2}\text{ by Jensen's inequality}\\
	&\le & \sum_{m=0}^{M-1} m^{k} |p_{nm} - p_m| + 2 (A k)^{k} \sqrt\frac{A}{ M}\text{ by \eqref{eq:moment} and Markov's inequality.}
	\end{eqnarray*}
	Letting $M$ and $n$ go to infinity, we obtain the convergence in moments.
\end{proof}

 \subsection{Convergence for the number of real roots}\label{sec:proof:uni} In this section, we prove \eqref{eq:uni}. In other words, we prove the following convergence of the number of roots $U_C$. Note that the random variables are not necessarily Gaussian here.
 
 The following generalizes Theorem \ref{theorem:distribution-moments} to non-Gaussian random variables.
  \begin{theorem}\label{theorem:moments:nongau-infty}
 	Let $C$ be a fixed positive number. For all $k\ge 0$, we have
 	$$\E N^{k}_{g_n}(U_C)\to \E N^{k}_{g_\infty}(U_C)$$
 	as $n\to\infty$.
 \end{theorem}

Let $\tilde g_n$ be the version of $g_n$ when the random variables $\xi_i$ are iid standard Gaussian. By Theorem \ref{theorem:distribution-moments}, we have
\begin{equation}\label{eq:conv:gau}
\E N^{k}_{\tilde g_n}(U_C)\to \E N^{k}_{g_\infty}(U_C).
\end{equation}
We note that the same proof holds with $U_C$ replaced by $U_C+B(0, \alpha)$.
\begin{proof} Note that the number of roots of $g_n$ in $U_C$ is the same as the number of roots in the original function $f_n$ in the set $\zeta + \frac{1}{n} \zeta U_C$, by \eqref{def:gn}.  Let $\varphi$ be a test function approximating the indicator of $(U_C)^{k}$, in particular, we let $\varphi$ be a smooth function such that
\begin{eqnarray}
\textbf{1}_{(U_C)^{k}}\le \varphi \le  \textbf{1}_{(U_C+B(0, \alpha))^{k}} \label{eq:varphi}
\end{eqnarray}
and $|\triangledown^{a}\varphi(z)|\ll 1$ for all multi-indices $a$ with $0\le |a|\le 2k+4$.

By \cite[Theorem 2.4]{DOV} \footnote{or perhaps a slightly readable \cite[Theorem 4.3]{nguyenvurandomfunction17} which was written for the Kac polynomial but it holds also for the derivatives of the Kac polynomial.} applied to the function $G = \varphi$ and the centers $z_1=\dots= z_k = \zeta$, we get
\begin{eqnarray}
&&\left |\E \sum_{\zeta_{i_1}, \dots, \zeta_{i_k}\in \mathcal Z(g_n)} \varphi(n(\zeta_{i_1}/\zeta-1), \dots, n(\zeta_{i_1}/\zeta-1))\right . \notag\\
&&\left .- \E \sum_{\zeta_{i_1}, \dots, \zeta_{i_k}\in \mathcal Z(\tilde g_n)} \varphi(n(\zeta_{i_1}/\zeta-1), \dots, n(\zeta_{i_1}/\zeta-1))\right |
\ll  n^{-c}\notag
\end{eqnarray}
where $c>0$ is a small constant.  Here, we note that the transformation $z:= n(\zeta_{i}/\zeta-1)$ is just the inverse of the rescaling map $\zeta_i=\zeta + \frac{1}{n} \zeta z$ that brings the neighborhood of $\zeta$ to $U_C$. We note that when $\varphi$ is replaced by $\textbf{1}_{(U_C)^{k}}$, the term under the expectation becomes $N^{k}(U_C)$. So, we have from \eqref{eq:varphi} that
$$\E N^{k}_{g_n}(U_C) \le \E N^{k}_{\tilde g_n}(U_C +B(0, \alpha)) + O_{\alpha}(n^{-c}).$$
Using \eqref{eq:conv:gau}, we obtain
$$\limsup_{n\to\infty}\E N^{k}_{g_n}(U_C) \le \limsup_{n\to\infty} \E N^{k}_{\tilde g_n}(U_C +B(0, \alpha)) = \E N^{k}_{g_\infty}(U_C+B(0, \alpha)).$$
Sending $\alpha$ to 0, we obtain 
$$\limsup_{n\to\infty}\E N^{k}_{g_n}(U_C) \le \limsup_{\alpha\to 0}\E N^{k}_{g_\infty}(U_C+B(0, \alpha)) = \E N^{k}_{g_\infty}(U_C)$$
where the last equality follows from the dominated convergence theorem, knowing that $\E N^{k}_{\tilde g_\infty}(U_C+B(0, \alpha))<\infty$ for some $\alpha>0$ (by \eqref{eq:moment}).
Similarly, we get the reverse direction and conclude the proof.
\end{proof}

\subsection{Upper bound the hole radius for $g_{n}$}\label{sec:proof:Puni} In this section, we show the following theorem.
\begin{theorem}
	 The random variables $N_{g_n}(U_C)$ converges to $N_{g_{\infty}}(U_C)$ in distribution as $n\to\infty$. In particular, we have \eqref{eq:conv:nu}:
	 $$\P(N_{g_n}(U_C)=0)\to \P(N_{g_\infty}(U_C)=0).$$
\end{theorem}
Here, we recall that since the random variables $N_{g_n}(U_C)$ are discrete random variables supported on $\N$, convergence in distribution means convergence of the probability density $\P(N_{g_n}(U_C)=i)$, as $i$ varies. 
\begin{proof} By Theorem \ref{theorem:moments:nongau-infty}, $N_{g_n(U_C)}$ converges to $N_{g_\infty}(U_C)$ in moments. By \ref{eq:moment} and the Carleman's criteria, $N_{g_\infty} (U_C)$ is uniquely determined by its moments. Thus, we infer that $N_{g_n(U_C)}$ converges to $N_{g_\infty}(U_C)$ in distribution.
\end{proof}

\section{Proof of Theorem \ref{thm:main}: Lower bound}\label{sec:proof:lb}

We want to show that there exists $c$ such that
\begin{equation}\label{key}
\E(N_{f_n}(B(\zeta, c/n))) \le \ep.
\end{equation}
Without loss of generality, we assume that $\ep<1/100$ and $c< \ep$.

The first step is to reduce to the Gaussian case, via universality results. Consider the Gaussian version of $f_n$,  
$$\tilde f_n = \tilde f_{\rho, n} = \sum_{i=0}^{n} a_{i, \rho, n} \tilde \xi_i z^i$$
where $\tilde \xi_i$ are iid standard Gaussian.

Let $G$ be a smooth function such that approximates the indicator of the ball, or more specifically, $\textbf{1}_{B(\zeta, c/n)}\le G\le \textbf{1}_{B(\zeta, 2c/n)}$ and $||\triangledown^{a}G||_{\infty} = O(n^{a})$ for all $a\le 3$. We now apply a universality property of $f_n$ established in \cite[Theorem 2.3]{DOV}. This theorem applied to the function $G$ states that the linear statistics $\E \sum_{w\in \mathcal Z(f_n)} G(w)$ is universal, i.e.,
$$\E \sum_{w\in \mathcal Z(f_n)} G(w) - \E \sum_{w\in \mathcal Z(\tilde f_n)} G(w)\ll n^{-\gamma}$$
for a constant $\gamma$ independent of $n$ and $\zeta$.

Using this, we obtain
\begin{eqnarray}
\E(N_{f_n}(B(\zeta, c/n))) &\le& \E \sum_{w\in \mathcal Z(f_n)} G(w) =  \E \sum_{w\in \mathcal Z(\tilde f_n)} G(w) + o(1) \notag\\
&\le& \E( N_{\tilde f_n}(B(\zeta, 2c/n))) +o(1).\notag
\end{eqnarray}
Thus, it suffices to  prove that
\begin{equation}\label{eq:lb}
\E( N_{\tilde f_n}(B(\zeta, 2c/n))) \le \ep/2.
\end{equation}

In other words, it suffices to prove for the Gaussian case. To this end, we let $B_c = B(0, 2c)$ and definite the functions $g_n$ and $g_\infty$ as before. We apply the Kac-Rice formula to $g_{\infty}$ to get
\begin{eqnarray}
\E N_{g_\infty}(B_{c}) &=& \int_{B_c}\rho_1(z). \notag
\end{eqnarray}
By \eqref{eq:devF}, we have for all $z\in B_c$, $ \rho_1(z) \ll 1.$
Thus, 
$$\E N_{g_\infty}(B_{c}) \ll c^{2}\le \ep/4.$$
By the same argument as for $U_C$ (noting that $B_c\subset U_C$ for small $c$ and for $C\ge c$), we obtain the same limit as in Theorem \ref{theorem:distribution-moments}. So, we get
$$\lim_{n\to\infty} \E N_{g_n}(B_{c}) = \E N_{g_\infty}(B_{c}) \le \ep/4.$$
So, by choosing $n$ to be sufficiently small, we obtain \eqref{eq:lb} as desired.

\section{Acknowledgment} We thank Manjunath Krishnapur for suggesting helpful references on correlation functions.
  
 \bibliographystyle{plain}
 \bibliography{polyref}
 
\end{document}